\documentclass[10pt]{amsart}
\usepackage{amsmath}
\usepackage{amsfonts}
\usepackage{enumerate}
\usepackage{amssymb}
\usepackage{amscd}

\DeclareMathOperator{\gr}{gr} 
 \DeclareMathOperator{\Aut}{Aut}
\DeclareMathOperator{\Hom}{Hom} 
\DeclareMathOperator{\End}{End}

\begin{document}
\theoremstyle{plain}
\newtheorem*{example}{Example}
\newtheorem*{cor}{Corollary}
\newtheorem*{conj}{Conjecture}
\newtheorem*{prop}{Proposition}
\newtheorem*{hyp}{Hypothesis}
\newtheorem*{thrm}{Theorem}
\newtheorem*{lem}{Lemma}
\theoremstyle{remark}
\newtheorem{defn}{Definition}
\newtheorem*{rem}{Remark}
\newtheorem*{rems}{Remarks}
\newtheorem{quest}{Question}

\title{Symplectic reflection algebras in positive characteristic}

\author{K.A. Brown and K. Changtong}

\address{Brown: Department of Mathematics,
University of Glasgow, Glasgow G12 8QW, UK}

\email{kab@maths.gla.ac.uk}

\address{Changtong:  Department of Mathematics,
Faculty of Science, Ubon Ratchathani University, Ubon Ratchathani, 34190, Thailand.}
\email{noikanok@gmail.com}

\begin{abstract} Basic properties of symplectic reflection
algebras over an algebraically closed field $k$ of positive
characteristic are laid out. These algebras are always finite
modules over their centres, in contrast to the situation in
characteristic 0. For the subclass of rational Cherednik algebras,
we determine the PI-degree and the Goldie rank, and show that the
Azumaya and smooth loci of the centre coincide.
\end{abstract}

\subjclass[2000]{}


\thanks{Part of the work described here was obtained in the PhD
thesis of the second author, for which she was awarded a PhD by the
University of Glasgow in 2006. Her PhD research was supported by a
grant from the Thai Royal Government. We thank Iain Gordon and
Catharina Stroppel for helpful comments. }

\maketitle

\section{Introduction} \label{Intro}
\subsection{} \label{start} Throughout this paper, $k$ will be an
algebraically closed field of characteristic $p$, where $p$ is an
odd prime. Symplectic reflection algebras over $\mathbb{C}$ were
introduced in \cite{EG}. The same definition makes sense over $k$,
and indeed symplectic reflection algebras over $k$ have been studied
in \cite{BFG} and \cite{L}. Let $(V, \omega)$ be a finite
dimensional symplectic vector space over $k$ and $S(V)$ its
symmetric algebra, and let $\Gamma$ be a finite subgroup of
$\mathrm{Sp}(V)$ with $\mathrm{char}k \nmid |\Gamma|.$ The
symplectic reflection algebra $H = H_{t,
\mathbf{c}}(V,\omega,\Gamma)$ is a deformation of the skew group
algebra $S(V) \ast \Gamma$; the precise definition is in
(\ref{def}). One can limit the study at once to the case where
$\Gamma$ is generated by its set $S$ of symplectic reflections: that
is, $s \in S$ if and only if $\mathrm{dim}_k\mathrm{im}(\mathrm{Id}
- s) = 2.$ In the definition, $t \in k$ and $\mathbf{c}:S
\longrightarrow k$ is a $\Gamma-$invariant function; $S(V) \ast
\Gamma$ [resp. $\mathcal{D}(V) \ast \Gamma$] corresponds to the case
where $\mathbf{c}$ is the zero map and $t=0$ [resp. $t=1$].

\subsection{} \label{split} Recall that Weyl algebras in positive
characteristic are finite modules over their centres; see Lemma
\ref{weyl}. In a parallel fashion, all symplectic reflection
algebras over $k$ are finite modules over their centres by a result
of Etingof \cite[Appendix 10]{BFG}. In contrast, over $\mathbb{C}$,
when $t \neq 0,$ $Z(H_{t,\mathbf{c}}) = \mathbb{C}$,
\cite[Proposition 7.2(2)]{BG}. When $t=0$ the theory over $k$
appears to be essentially the same as over $\mathbb{C}.$ Thus the
focus here will be on the case where $t$ is non-zero; in fact, after
re-scaling, we may then assume that $t=1$.

\subsection{} \label{results} In $\S$\ref{back} minor adjustments
to the characteristic 0 arguments suffice to show that $H =
H_{1,\mathbf{c}}$ is a prime noetherian $k-$algebra with excellent
homological properties - namely, it is Auslander-regular and
Cohen-Macaulay.

Let $e \in k\Gamma \subseteq H$ be the symmetrising idempotent
$\frac{1}{|\Gamma|}\sum_{\gamma \in \Gamma}\gamma$. We show in
$\S$\ref{sphere} that the symmetrising algebra $eHe$ is a
noetherian domain, a maximal order with - again - good homological
properties. Just as in characteristic 0, the Satake homomorphism
defines an isomorphism between $Z(H)$ and $Z(eHe)$. The algebras
$H$ and $eHe$ are connected by a Morita context, whose details are
laid out in (\ref{Morita}) and (\ref{return}).

\subsection{}\label{Cher} In $\S\S$\ref{Cherone} and \ref{Cherrep}
we specialise to the case where $H$ is a rational Cherednik algebra,
and prove our main results. That is, $\Gamma$ acts on a vector space
$\mathfrak{h}$ of dimension $n$, with $\Gamma$ generated by
pseudo-reflections for this action, and $V = \mathfrak{h} \oplus
\mathfrak{h}^*$ with the canonical $\Gamma-$invariant symplectic
form, denoted by $\omega$.

\begin{thrm} Let $k, \,\mathfrak{h}, \, n, \, \Gamma, \,\omega$ and $H$ be as
above. Denote the centre of $H$ by $Z(H)$.
\begin{enumerate}
\item  $H$ is a free module of rank
$p^{2n}|\Gamma|^3$ over the central subalgebra
$$ Z_0 := (S(\mathfrak{h})^{\operatorname{p}})^{\Gamma} \otimes
(S(\mathfrak{h}^*)^{\operatorname{p}})^{\Gamma},$$ where
$(-)^{\operatorname{p}}$ denotes the Frobenius map. \item The Goldie
rank of $H$ is $|\Gamma|$. \item The PI-degree of $H$ is $
p^n|\Gamma|. $ \item Every irreducible $H$-module of maximal
$k$-dimension $p^n|\Gamma|$ is a regular $k\Gamma$-module of rank
$p^n$. \item The Azumaya locus of $H$ is equal to the smooth locus
of $\mathrm{Maxspec}(Z(H))$: that is, for a maximal ideal
$\mathfrak{m}$ of $Z(H)$, $\mathfrak{m}$ annihilates a simple
$H-$module of the maximal possible $k-$dimension $p^n|\Gamma|$ if
and only if $Z(H)_{\mathfrak{m}}$ has finite global dimension.
\end{enumerate}
\end{thrm}

\subsection{} In the last section, $\S$\ref{qn}, we gather
together some questions and conjectures arising from this work.

\section{Definitions and background}\label{back}
\subsection{} \label{def} Recall the hypothesis on $k$ from (\ref{start}.
When we want to emphasise that a particular result is valid over all
algebraically closed fields $K$ such that \begin{equation} 2|\Gamma|
\textit{ is a unit in } K \label{field}\end{equation},we will always
denote the ground field by $K$. All fields will be assumed to
satisfy (\ref{field}). The ingredients needed for the construction
of a symplectic reflection algebra are a finite dimensional
symplectic $K-$vector space $(V,\omega)$ and a finite group $\Gamma$
of symplectic automorphisms of $V.$ An element $\gamma \in \Gamma$
is called a \emph{symplectic reflection} (on $V$) if the rank of
$\mathrm{Id}- \gamma$ is 2. Let $S$ denote the set of symplectic
reflections in $\Gamma$ and let $\mathrm{\bf{c}} : S \longrightarrow
K : s \mapsto c_s$ be a $\Gamma-$invariant map. Let $t \in K.$ The
\emph{symplectic reflection algebra} $H_{t,\mathrm{\bf{c}}}$
\cite[page 245]{EG} is the deformation of the skew group algebra
$S(V) \ast \Gamma$ obtained by replacing the commutativity relations
$xy - yx = 0$ defining the polynomial algebra $S(V)$ by new
relations
\begin{eqnarray} xy - yx = t \omega (x,y)1_{\Gamma} - \sum_{s \in
S}c_s \omega_s (x,y)s, \label{tight} \end{eqnarray} for all $x,y \in
V$. Here, $w_s$ is the skew-symmetric form on $V$ which has
$\mathrm{ker} (\mathrm{Id}- s)$ in its radical, and coincides with
$\omega$ on $\mathrm{im}(\mathrm{Id} - s).$ Thus, writing $T(V)$ for
the tensor algebra of $V$, $H_{t,\mathrm{\bf{c}}}$ is the factor of
the skew group algebra $T(V)\ast \Gamma$ by the ideal generated by
the elements (\ref{tight}). Throughout this paper we will denote the
dimension of $V$ by $2n$, and we'll assume for convenience that
$\Gamma = \langle S \rangle$.

\subsection{PBW theorem and homological consequences} \label{prop} It's clear from the
definition that if $H_{t,\mathrm{\bf{c}}} =
H_{t,\mathrm{\bf{c}}}(V,\omega,\Gamma)$ is a symplectic reflection
algebra, then $H_{t,\mathrm{\bf{c}}}$ is an $\mathbb{N}-$filtered
algebra with filtration $\mathcal{F}_0 := K\Gamma$, $\mathcal{F}_1
:= K\Gamma + K\Gamma V,$ and $\mathcal{F}_i := \mathcal{F}_1^i$ for
$i \geq 1.$ It's also immediate from the relations (\ref{tight})
that there is an algebra epimorphism $$\pi : S(V) \ast \Gamma
\longrightarrow \mathrm{gr}_{\mathcal{F}}(H_{t,\mathrm{\bf{c}}});$$
the starting point for the study of symplectic reflection algebras
is \cite[Theorem 1.3]{EG}, the
``Poincar$\acute{\mathrm{e}}$-Birkhoff-Witt theorem'', which asserts
that \begin{eqnarray}\pi \textit{ is an isomorphism.}\label{PBW}
\end{eqnarray}

Equivalently, the identity map on $V$ and $\Gamma$ extends to a
vector space isomorphism between $H_{t,\mathrm{\bf{c}}}$ and
$S(V)\otimes \Gamma.$ Some of the consequences flowing from this
fact are gathered in the theorem below, for which we need to recall
a few definitions. A noetherian algebra $A$ is
\emph{Auslander-Gorenstein} if $A$ has finite left injective
dimension $d$, and for every finitely generated left $A-$module $M$,
every non-negative integer $i$ and every submodule $N$ of
$\mathrm{Ext}_A^i(M,A)$, we have $\mathrm{Ext}_A^j(N,A) = 0$ for all
$j < i;$ and the same conditions hold with left replaced by right
throughout. By a theorem of Zaks \cite{Za}, if a noetherian algebra
has finite right and left injective dimensions, then these are
equal. If $A$ is Auslander-Gorenstein and has finite global
dimension, then $A$ is called \emph{Auslander-regular}, (and then
necessarily $\mathrm{gldim}(A) = d$). The \emph{grade} $j_A(M) \in
\mathbb{N} \cup \{+\infty \}$ of a non-zero finitely generated
$A-$module $M$ is the least integer $j$ such that $\mathrm{Ext}_A^j
(M,A)$ is non-zero. If $A$ has finite Gel'fand-Kirillov dimension,
then $A$ is \emph{Cohen-Macaulay} if
\begin{eqnarray} j_A(M) + \mathrm{GK-dim}(M) = \mathrm{GK-dim}(A)\label{CM}\end{eqnarray} for every
non-zero finitely generated left or right $A-$module $M$. We recall
that if $A$ has a positive filtration such that $\mathrm{gr}(A)$ is
Auslander-Gorenstein, and $M$ is a finitely generated $A-$module
endowed with a standard filtration, then
$$ j_A(M) = j_{\mathrm{gr}(A)}(\mathrm{gr}(M))$$ by \cite[(1) in the proof of Theorem
3.9]{Bj}; so that, since $$\mathrm{GK-dim}(M) =
\mathrm{GK-dim}_{\mathrm{gr}(A)}(\mathrm{gr}(M))$$ by
\cite[Proposition 8.6.5]{MR}, \begin{eqnarray} \textit{if }
\mathrm{gr}(A) \textit{ is Cohen-Macaulay, then so also is }
A.\label{CMgr}
\end{eqnarray}
The following basic facts are well known, and were for the most
part proved for algebras over $\mathbb{C}$ in \cite{EG}; the
proofs over a general field $k$ are identical, depending on
filtered graded techniques and the corresponding statements for
$S(V) \ast \Gamma$, once one knows from (\ref{PBW}) that
$\mathrm{gr}_{\mathcal{F}}(H_{t,\bf{c}}) \cong S(V)\ast \Gamma.$

\begin{thrm} Let $K$ be an arbitrary field, and let $H_{t,\mathrm{\bf{c}}} = H_{t, \mathrm{\bf{c}}}(V,\omega,\Gamma)$ be a symplectic
reflection algebra, with $\mathrm{dim}_K (V) = 2n.$ Then the
following hold.
\begin{enumerate}
\item $H_{t,\mathrm{\bf{c}}}$ is a prime noetherian algebra. \item
$H_{t,\mathrm{\bf{c}}}$ has finite Gel'fand-Kirillov dimension,
$\mathrm{GK}-\mathrm{dim}(H_{t,\mathrm{\bf{c}}}) = 2n.$ \item
$H_{t,\mathrm{\bf{c}}}$ is Auslander-regular and Cohen-Macaulay,
with $\mathrm{gldim}(H_{t,\mathrm{\bf{c}}}) \leq 2n.$
\end{enumerate}
\end{thrm}
\begin{proof} (i) \cite[Theorem 1.6.9]{MR} and \cite[Ex.6.6]{Co}.

(ii) \cite[Proposition 8.6.5]{MR} and \cite[Proposition 5.5]{KL}.

(iii) First, $S(V)\ast\Gamma$ is Auslander-regular by \cite{Yi}.
This property lifts to $H_{t, \mathrm{\bf{c}}}$ by \cite[Theorem
3.9 and the Remark following]{Bj}. The Cohen-Macaulay property is
dealt with in the discussion preceding the theorem.
\end{proof}

\subsection{The centre}\label{centre} There is a fundamental
dichotomy in the theory determined by whether or not the parameter
$t$ is zero. Since $H_{t \mathrm{\bf{c}}} \cong H_{\lambda t,
\lambda \mathrm{\bf{c}}}$ for $0 \neq \lambda \in k,$ we need
consider only the cases $t = 0$ and $t = 1.$

\begin{thrm} Let $H_{t,\mathrm{\bf{c}}} = H_{t, \mathrm{\bf{c}}}(V,\omega,\Gamma)$ be a symplectic
reflection algebra over an arbitrary field $K$.
\begin{enumerate}
\item (\cite[Theorem 3.1]{EG}, \cite[Proposition 7.2(2)]{BG})
Suppose that $K$ has characteristic 0. If $t=0,$ then $H_{
0,\mathrm{\bf{c}}}$ is a finite module over its centre
$Z(H_{0,\mathbf{c}})$, which is Gorenstein. If $t = 1,$ then
$Z(H_{1,\mathrm{\bf{c}}}) = K.$
\item  Suppose for the rest of the theorem that $K=k$. \begin{enumerate}\item $(\mathrm{Etingof},$
\cite[Appendix 10]{BFG}$)$ $H_{t,\mathrm{\bf{c}}}$ is a finite
module over its centre for all values of $t.$ With respect to the
filtration $\{F_i\}$ of (\ref{prop}), the associated graded algebra
of $Z = Z(H_{1,\mathbf{c}})$ is
$(S(V)^{\operatorname{p}})^{\Gamma},$ where $(-)^{\operatorname{p}}$
denotes the Frobenius homomorphism. \item $Z(H_{t, \mathbf{c}})$ is
Gorenstein.
\end{enumerate}
\end{enumerate}
\end{thrm}

\begin{proof} The statements for $t=0$ are given by \cite[Theorem 3.3]{EG},
for all ground fields. The only other claim which is not as stated
in the cited references is (ii)(b) for non-zero $t$. For this, by
\cite[Theorem 3.9]{Bj} it's enough to prove that $\mathrm{gr}(Z(H))$
is Gorenstein. By the second statement in (ii), bearing in mind that
$(S(V)^{\operatorname{p}})^{\Gamma}=(S(V)^{\Gamma})^{\operatorname{p}},$
$\mathrm{gr}(Z(H))$ is isomorphic to $S(V)^{\Gamma}.$ Since $\Gamma
\subseteq SL(V)$, $S(V)^{\Gamma}$ is Gorenstein by Watanabe's
theorem \cite{Wa}.
\end{proof}

It should be noted that the uncertainty over the precise value of
the global dimension of $H_{t,\mathrm{\bf{c}}}$ noted in Theorem
\ref{prop}(iii) disappears when its centre is big. For in this
case every irreducible $H_{t,\mathrm{\bf{c}}}-$module $M$ is a
finite dimensional $K$-vector space, so that $\mathrm{GK-dim}(M) =
0.$ Therefore, by (\ref{CM}) and Theorem \ref{prop}(ii),
$\mathrm{pr.dim}(M)=2n,$ and so we deduce the

\begin{thrm} Let $K$ be an arbitrary field satisfying (\ref{field}), and suppose that the symplectic reflection algebra $H
=H_{t,\mathrm{\bf{c}}}(V, \omega, \Gamma)$ is a finite module over
its centre. Then
$$ \mathrm{gldim}(H) = \mathrm{Krull \; dim}(H)
= \mathrm{GK-dim}(H) = \mathrm{dim}(V). $$
\end{thrm}

\section{Interplay with the spherical subalgebra}\label{sphere}
For the rest of the paper our ground field will be $k$, and we
focus on the case where $k$ has positive characteristic and the
theory deviates from that over characteristic 0; that is, we shall
assume that the parameter $t$ is non-zero, so that, after
re-scaling, we can assume that $t=1.$ So from now on $H$ will
denote a symplectic reflection $k$-algebra $H :=
H_{1,\mathbf{c}}(V, \omega,\Gamma).$
\subsection{Basic properties}\label{base} Recall that, following
\cite{EG}, the \emph{symmetrising idempotent} $e$ of $H$ is $e =
\frac{1}{|\Gamma|}\sum_{\gamma \in \Gamma} \gamma,$ and $eHe$ is
the \emph{spherical subalgebra} of $H.$ (Note that it is
\emph{not} a subalgebra of $H$ in the usual sense of the word,
since it doesn't contain $1_H.$) The filtration $\{\mathcal{F}_i
\}$ of $H$ induces a filtration $\{e\mathcal{F}e \}$ of $eHe$ with
$i^{\mathrm{th}}$ subspace $\mathcal{F}_iH \cap eHe =
e\mathcal{F}_iH e$. It follows that
\begin{eqnarray} \mathrm{gr}_{e\mathcal{F}e}(eHe) = e\,\mathrm{gr}_\mathcal{F}(H)e = eS(V)\ast \Gamma e \cong
S(V)^{\Gamma}.\label{sphergrad}\end{eqnarray} Standard
filtered-graded methods can therefore be used to deduce the
following theorem. For (v), we recall the definition of a maximal
order. Suppose that the noetherian algebra $A$ is prime, so that
it has a simple artinian quotient ring $Q$ by Goldie's theorem,
\cite[Theorem 2.3.6]{MR}. There is an equivalence relation on
orders in $Q$ defined by: $R\sim S$ if and only if there exist
units $a,b,c,d$ of $Q$ with $aRb \subseteq S$ and $cSd \subseteq
R$, \cite[5.1.1, 3.1.9]{MR}. Then $A$ is a \emph{maximal order} if
it is a maximal member of its equivalence class of orders in $Q$.
A commutative noetherian integral domain is a maximal order if and
only if it is integrally closed, \cite[Lemma 5.3.2]{MR}.

\begin{thrm}\label{spherprop}
\begin{enumerate}
\item $eHe$ is a finitely generated module over its centre $eZe.$
\item $eHe$ is an affine noetherian domain. \item $He$ is a
finitely generated right $eHe-$module.\item $eHe$ is
Auslander-Gorenstein and Cohen-Macaulay. \item $eHe$ is a maximal
order.
\end{enumerate}
\end{thrm}
\begin{proof} (i) is clear from Theorem \ref{centre}(ii)(a). Note next that $S(V)^{\Gamma}$
is affine and $S(V)$ is a finitely generated $S(V)^{\Gamma}-$module
by the Hilbert-Noether theorem, \cite[Theorem 1.3.1]{Be}; thus $eHe$
is affine by (\ref{sphergrad}), and since $\mathrm{gr}(He) \cong
S(V)$ (iii) follows also. The rest of (ii) follows from
(\ref{sphergrad}) and \cite[Theorems 1.6.9]{MR} and \cite[Ex.
6.6]{Co}.

To prove (iv) we first observe that $S(V)^{\Gamma}$ is Gorenstein by
Watanabe's theorem \cite[Theorem 1]{Wa}, since $\Gamma \subseteq
\mathrm{Sp}(V) \subseteq SL(V)$. The Auslander-Gorenstein condition
thus follows from \cite[Chapter II, Proposition 2.2.1]{LiVO}.
Moreover, since $S(V)^{\Gamma}$ is Cohen-Macaulay, either by the
Hochster-Eagon theorem \cite[Theorem 4.3.6]{Be}, or since
commutative Gorenstein rings are Cohen-Macaulay, it follows from
(\ref{CMgr}) in $\S$\ref{prop} that $eHe$ is Cohen-Macaulay. For (v)
it's enough, in view of \cite[Theorem 5.1.6]{MR} and
(\ref{sphergrad}), to note that the fixed ring $S(V)^{\Gamma}$ is
integrally closed, \cite[Proposition 1.1.1]{Be}.\end{proof}

\subsection{The Morita context} \label{Morita} In this paragraph we prove a version of \cite[Theorem 1.5]{EG}
for symplectic reflection algebras of positive characteristic. For
the most part the proofs follow the proof of \cite[Theorem 1.5]{EG},
so we only indicate the key steps. Part (ii) is an example of a
\emph{double centralizer property} in the spirit of \cite{KSX}.

\begin{thrm}
\begin{enumerate}
\item $ \End_H(He) \quad \cong \quad eHe.$ \item $\End_{eHe}(He)
\quad \cong \quad H.$ \item $\mathrm{Hom}_{eHe}(He_|,eHe_|) \cong
eH,$ and $\mathrm{Hom}_{eHe}({}_|eH,{}_|eHe) \cong He.$ In
particular, $He$ is a reflexive right $eHe-$module, and $eH$ is left
$eHe-$reflexive.
\end{enumerate}
\end{thrm}

\begin{proof} (Sketch) Of course (i) is clear.

(ii) {\bf Step 1:} \emph{Construction of an ascending filtration
on} $\mathrm{End}_{eHe}(He).$ Since $He$ is filtered by
$\{\mathcal{F}_iHe\}$, we can choose a finite set of generators
$u_1, \ldots , u_l$ of $\mathrm{gr}_{\mathcal{F}}(He),$ whose
lifts $\{\tilde{u}_1, \ldots , \tilde{u}_l\}$ generate the
$eHe-$module $He.$ Let $\tilde{u}_i \in \mathcal{F}_{d_i}(He).$ We
can assume that $\mathcal{F}_n (He) =
\sum_{j=1}^{l}\tilde{u}_j\mathcal{F}_{n-d_j}(eHe)$ for all $n>>0.$
Hence, for $f \in \mathrm{End}_{eHe}(He),$ there is a non-negative
integer $m$ such that $f(\mathcal{F}_n He \subseteq
\mathcal{F}_{n+m}He$ for all $n \geq 0.$ In this case we say that
$f \in \mathcal{F}_m\mathrm{End}$; in this way
$\mathrm{End}_{eHe}(He)$ acquires an ascending filtration.

{\bf Step 2:} \emph{The map} $\eta: H \longrightarrow
\mathrm{End}_{eHe}(He).$ For $h \in H,$ let $\eta(h)$ denote left
multiplication of $He$ by $h$, so it's clear that $\eta: H
\longrightarrow \mathrm{End}_{eHe}(He)$ is an algebra homomorphism.
We have to prove that $\eta$ is an isomorphism. Since $\eta$
preserves the filtrations it induces a homomorphism
$\mathrm{gr}(\eta)$ of the corresponding graded algebras, and it's
enough to prove that the latter is an isomorphism. Consider then the
composition
\begin{eqnarray} S(V)\ast \Gamma \cong \mathrm{gr}(H)
\stackrel{\mathrm{gr}(\eta)}\longrightarrow & \mathrm{gr}(\mathrm{End}_{eHe}(He)) &\\
\nonumber \stackrel{j}\longrightarrow &
\mathrm{End}_{\mathrm{gr}(eHe)}(\mathrm{gr}(He)) & \cong
\mathrm{End}_{S(V)^{\Gamma}}(S(V)), \label{compmap}\end{eqnarray}
where $j$ is the obvious homomorphism. Injectivity follows because,
after tensoring with the quotient field $Q(S(V)^{\Gamma})$ of
$S(V)^{\Gamma})$, the induced map $\psi := \mathrm{Id}
\otimes_{S(V)^{\Gamma}} (j \circ \mathrm{gr}(\eta))$ is easily seen
to be injective.

{\bf Step 3:} \emph{Surjectivity of }$\eta.$ We prove this by
demonstrating surjectivity of $j \circ \mathrm{gr}(\eta),$ and
claim first that
$$ \psi \textit{ is surjective.} $$ This follows because $\psi$ is
an injective homomorphism from $Q(S(V))\ast \Gamma$ to\\
$\mathrm{End}_{Q(S(V)^{\Gamma})}(Q(S(V))),$ and these two
$Q(S(V)^{\Gamma})-$vector spaces have the same dimension, namely
$|\Gamma |^2.$ Hence $j \circ \mathrm{gr}(\eta)(S(V)\ast \Gamma)
\subseteq \mathrm{End}_{S(V)^{\Gamma}}(S(V)),$ and these two
algebras have the same simple artinian quotient ring, namely
$\mathrm{End}_{Q(S(V)^{\Gamma})}(Q(S(V)))$. But note that, since
$S(V)$ is a finitely generated $S(V)^{\Gamma}-$module, so too is
\newline $ \mathrm{End}_{S(V)^{\Gamma}}(S(V)).$

\emph{ A fortiori}, $\mathrm{End}_{S(V)^{\Gamma}}(S(V))$ is a
finitely generated module over its subalgebra $j \circ
\mathrm{gr}(\eta)(S(V)\ast \Gamma)$. However this last algebra is
a maximal order, by \cite[Theorem 4.6]{Ma}, so the inclusion of
algebras must be an equality as required.

(iii) This follows from (ii) just as in \cite{EG} - that is, one
confirms using (ii) that the map from $eH$ to
$\mathrm{Hom}_{eHe}(He,eHe)$ induced by left multiplication is an
isomorphism of right $eHe-$modules. The second statement follows
symmetrically.
\end{proof}

\noindent {\bf Corollary A.} $H$ \emph{is a maximal order.}

\begin{proof} Suppose that $T$ is an order in $Q(H)$ with $H
\subseteq T$ and $T$ equivalent to $H$. By \cite[Lemma 3.1.10]{MR},
there is a non-zero ideal $I$ of $H$ with either $IT \subseteq H$ or
$TI \subseteq H.$ Suppose the former holds. Then $0 \neq
eIe\vartriangleleft eHe,$ since $H$ is prime, with $$(eIe)(eTe)
\subseteq eITe \subseteq eHe. $$ Thus $eTe$ is an order in $eQ(H)e =
Q(eHe),$ which contains $eHe$ and is equivalent to it. So, by
Theorem \ref{base}(v),
\begin{eqnarray} eTe = eHe. \label{car} \end{eqnarray}
Note that, if instead $TI \subseteq H,$ we can still arrive at
(\ref{car}). Bearing in mind the identifications of Theorem
\ref{Morita}(iii), (\ref{car}) shows that $Te \subseteq (eH)^*,$ and
so, by Theorem \ref{Morita}(iii), $$Te = He. $$ In other words, $T
\subseteq \mathrm{End}_{eHe}(He).$ From Theorem \ref{Morita}(ii) we
deduce that $T = H,$ as required.
\end{proof}

\begin{rems}  (i)The proof of the corollary works over fields of
characteristic 0, and is independent of the value of the parameter
$t$; in these other cases the result does not seem to have been
previously recorded.

(ii) One would naturally expect to prove the above corollary by
lifting the result using filtered-graded methods from the
corresponding result for the skew group algebra $S(V) \ast \Gamma
\cong \mathrm{gr}_F(H).$ However the relevant lifting theorem in the
literature, Chamarie's theorem \cite[Theorem 5.1.6]{MR}, requires
the algebras involved to be domains. While this defect can
presumably be rectified, it seems more efficient to proceed as
above.
\end{rems}

Below and in (\ref{return}) we need the concept of a
\emph{localizable} prime ideal $P$ of a noetherian ring $R$: this
means that the set $\mathcal{C}(P)$ of elements of $R$ whose images
{\it modulo} $P$ are not zero divisors forms a (right and left) Ore
set in $R$. When this happens, one can invert the elements of
$\mathcal{C}(P)$ to obtain the local ring $R_P$, a partial quotient
ring of the factor ring of $R$ by the ideal $I = \{ r \in R : cr = 0
\textit{ or } rc = 0 \textit{ for some } c \in \mathcal{C}(P) \}.$
When $R$ is a finite module over its centre $Z$, a prime ideal $P$
of $R$ is localizable when (and, in fact, only when) it is the
\emph{unique} prime of $R$ lying over $P \cap Z$, and one sees
easily that in this case  we can form $R_P$ by inverting the
elements of $Z \setminus (P \cap Z).$ For background on these ideas,
see, for example, \cite{GW}.

Standard Morita theory \cite[Proposition 3.5.6]{MR} applied to the
theorem above tells us that $H$ is Morita equivalent to $eHe$
exactly when $HeH = H.$ More precisely, the size of $HeH$ indicates
how close $H$ and $eHe$ are to being Morita equivalent. In this
connection, we thus have:

\begin{lem} Let $e$ be the symmetrising idempotent of $H$. Let
$\mathfrak{p}$ be a prime ideal of $Z:= Z(H)$ with $HeH \cap Z
\subseteq \mathfrak{p}.$ Then $\mathfrak{p}$ has height at least 2.
\end{lem}

\begin{proof} Suppose the result is false: so there is a prime
$\mathfrak{p}$ of $Z$ of height 1, with \begin{eqnarray} HeH \cap Z
\subseteq \mathfrak{p}.\label{inside}\end{eqnarray} Thus, we can
localize at $\mathfrak{p}$ by inverting the elements $Z \setminus
\mathfrak{p},$ to obtain the ring $H_{\mathfrak{p}}.$ Since $H$ is a
prime maximal order and is a finite module over its centre, by
Theorems \ref{prop}(i) and \ref{centre}(ii), and Corollary
\ref{Morita}A, all its height one primes are localizable, by
\cite[Propositions II.2.2 and II.2.6, and
Theor$\grave{\mathrm{e}}$me IV.2.15]{MaRa}. Equivalently, by
\cite[Theorem III.9.2]{BGb} or \cite[Theorem 7]{Mu}, there is a
\emph{unique} height one prime of $H$ lying over $\mathfrak{p}$ -
let's call it $P.$ Thus $H_{\mathfrak{p}}$ is a local ring with
Jacobson radical $PH_{\mathfrak{p}}.$ Now (\ref{inside}) ensures
that $H_{\mathfrak{p}}eH_{\mathfrak{p}}$ is a proper ideal of
$H_{\mathfrak{p}}$, so that $e \in PH_{\mathfrak{p}},$ the Jacobson
radical of $H_{\mathfrak{p}}.$ Hence $1-e$ is a unit in
$H_{\mathfrak{p}}$, a contradiction. So the result is proved.
\end{proof}

\noindent {\bf Corollary B.} \emph{Let} $\mathfrak{p}$ \emph{be any
prime ideal of $Z$ such that $HeH \cap Z$ is not contained in
$\mathfrak{p}.$ Then $H_{\mathfrak{p}}$ is Morita equivalent to
$eH_{\mathfrak{p}}e.$ In particular, this is true for every prime
ideal of $Z$ of height 1.}

\subsection{The Satake homomorphism} \label{satake} The proof of
the next result also follows the corresponding argument used to
prove \cite[Theorem 3.1]{EG}; but note the important difference
that, when $t=0,$ $eHe$ is commutative.

\begin{thrm} The map $\theta : H \rightarrow eHe : u \mapsto eue$
is an algebra homomorphism when restricted to the centre $Z$ of
$H$, and maps $Z$ isomorphically to the centre of $eHe.$
\end{thrm}
\begin{proof} (Sketch) It's obvious that the restriction of $\theta$ to $Z$
is an algebra homomorphism. We define an inverse map $\xi: Z(eHe)
\longrightarrow Z$ as follows. Let $eae \in Z(eHe)$, so that right
multiplication of $He$ by $eae$ defines a right $eHe-$module
endomorphism $r_{eae}$ of $He$. By Theorem \ref{Morita}(ii) this
endomorphism must be induced by left multiplication of $He$ by an
element $\xi(eae)$ of $H$; moreover, since $r_{eae}$ commutes with
the left multiplications of $He$ by the elements of $H$, $\xi(eae)
\in Z.$ It's now clear that $\xi$ is an algebra homomorphism, and
easy to check that it is inverse to $\theta_{|Z}.$
\end{proof}

\subsection{The Morita context revisited}
\label{return} An important theme in the study of those symplectic
reflection algebras $H$ which are finite modules over their centres
$Z$ has been the determination of the groups and parameter values
(if any) for which the centre is smooth \cite{EG}, \cite{G},
\cite{Mar}, \cite{Bel}.

Recall \cite[Theorem III.1.6]{BGb} that a prime $k-$algebra $A$,
finitely generated as a module over over its affine centre $Z$, is
\emph{Azumaya} over $Z$ if (and only if) all the irreducible
$A-$modules have the same vector space dimension (which is then
necessarily equal to the PI-degree of $A$. As we'll see in
(\ref{Azumaya}), the smoothness of $Z(H)$ is closely related to the
question of when $H$ is Azumaya over $Z$. Regarding this latter
question, we have

\begin{thrm} Let $P$ be a prime ideal of $H$, and let $\mathfrak{p}$
be the prime $P \cap Z$ of $Z:=Z(H)$. Consider the following
statements:
\begin{enumerate}
\item $\mathfrak{p}$ is in the Azumaya locus.
\item There exists a positive integer $s$ such that
$$ H_{\mathfrak{p}} \cong M_s(eH_{\mathfrak{p}}e),$$
and $eH_{\mathfrak{p}}e$ is a local ring with Jacobson radical
$e\mathfrak{p}H_{\mathfrak{p}}e.$
\item $P$ is a localizable ideal of $H$.
\item $P$ is the unique prime ideal of $H$ lying over
$\mathfrak{p}.$
\item $HeH \cap Z \nsubseteqq \mathfrak{p},$ and $ePe$ is the unique
prime ideal of $eHe$ lying over $e\mathfrak{p}e.$
\item There exists a positive integer $s$ such that
$$ H_{\mathfrak{p}} \cong M_s(eH_{\mathfrak{p}}e),$$
and $eH_{\mathfrak{p}}e$ is a local ring.
\end{enumerate}
Then (i) $\Leftrightarrow$ (ii) $\Rightarrow$ (iii)
$\Leftrightarrow$ (iv) $\Leftrightarrow$ (v) $\Leftrightarrow$ (vi).
\end{thrm}

\begin{proof} $(i)\Rightarrow (iii):$ This is clear, since, if $\mathfrak{p}$ is Azumaya in $H$,
then inverting $Z \setminus \mathfrak{p}$ yields the local ring
$H_P.$

$(iii) \Leftrightarrow (iv): $ This is M\"{u}ller's theorem,
\cite[Theorem 7]{Mu} or \cite[Theorem III.9.2]{BGb}.

$(iv) \Rightarrow (v):$ Assume (iv), which by the above is
equivalent to (iii). If $HeH \cap Z \subseteq \mathfrak{p},$ then
$H_PeH_P$ is a proper ideal of $H_P = H[Z \setminus
\mathfrak{p}]^{-1}$, and this is impossible just as in the proof of
Lemma \ref{Morita}. This proves the first part of (v); the second
part follows from Corollary \ref{Morita}, which shows that $eH_Pe =
eH_{\mathfrak{p}}e$ is local with Jacobson radical $ePH_Pe.$

$(v) \Rightarrow (vi):$ Assume (v). By Corollary \ref{Morita}
$H_{\mathfrak{p}}$ is Morita equivalent to $eH_{\mathfrak{p}}e$, and
the latter ring is local by (v). Now $H_{\mathfrak{p}}e$ is a
finitely generated projective right module over the local domain
$eH_{\mathfrak{p}}e$, by the Morita theorems and Theorem
\ref{prop}(ii) and (iii). Hence
$H_{\mathfrak{p}}e_{|eH_{\mathfrak{p}}e}$ is free of finite rank
$s$, and its endomorphism ring $H_{\mathfrak{p}}$ is thus isomorphic
to $M_s(eH_{\mathfrak{p}}e)$.

$(vi) \Rightarrow (iii):$ Assume (vi). Then $H_{\mathfrak{p}}$ is a
local ring in which $PH_{\mathfrak{p}}$ is a maximal ideal. Thus
$PH_{\mathfrak{p}}$ must be the Jacobson radical of
$H_{\mathfrak{p}}$, so $P$ is localizable, proving (iii).

$(ii) \Rightarrow (i):$ If (ii) holds, then $H_{\mathfrak{p}}$ is a
local ring whose Jacobson radical is generated by its intersection
with the centre. Thus $H_{\mathfrak{p}}$ is Azumaya, by
\cite[Theorem III.1.6]{BGb}.

$(i) \Rightarrow (ii):$ Assume (i). Then, as we have seen, (vi)
holds, and, writing $J(R)$ for that Jacobson radical of a ring $R$,
$$J(H_{\mathfrak{p}}) \cong J(M_s(eH_{\mathfrak{p}}e)) =
M_s(J(eH_{\mathfrak{p}}e))=M_s(e\mathfrak{p}H_{\mathfrak{p}}e),$$
since $\mathfrak{p}$ is Azumaya.
\end{proof}

\begin{rems} (i) We expect that (i) - (vi) in the above theorem
should be equivalent. This is true for all symplectic reflection
algebras (over any field) at $t=0$, by \cite[Theorem 1.7]{EG}. But
the proof depends crucially on the commutativity of $eHe$, in
particular \cite[Lemma 2.24]{EG}. When $k$ has positive
characteristic and $t=1$ we can't prove the equivalence of (i) -
(vi) even in the case of Cherednik algebras; the difficulty lies in
our inadequate understanding of the relation of the maximal order
$eHe$ to its centre.

(ii) We expect the integer $s$ of Theorem \ref{return}(vi) to be
$|\Gamma|$; we prove this in Theorem \ref{goldie} when $H$ is a
Cherednik algebra.

\end{rems}

\section{Cherednik algebras - structure}\label{Cherone}
\subsection{} \label{Cherdef} For the rest of the paper we shall assume that $H$ is a
\emph{rational Cherednik algebra} over an algebraically closed
field $k$ of positive characteristic $p.$ We fix a finite
dimensional $k-$vector space $\mathfrak{h}$ and a finite group
$\Gamma$ of automorphisms of $\mathfrak{h}$. We shall assume
throughout that $\Gamma$ is generated by pseudo-reflections for
its action on $\mathfrak{h}.$ Then $\Gamma$ acts on $V:=
\mathfrak{h} \oplus \mathfrak{h}^*,$ and this space admits a
canonical $\Gamma -$invariant symplectic form $\omega ,$ defined
by $\omega ((u,f),(x,g)) := g(u) - f(x)$ for $u,x \in
\mathfrak{h}$ and $f,g \in \mathfrak{h^*}.$ The set $S$ of
pseudo-reflections in $\Gamma$ is then the set of symplectic
reflections for $\Gamma$ acting on $V$, so - as in (\ref{def}) -
we can choose $t \in k$ and a $\Gamma -$invariant function
$\textbf{c}: S \longrightarrow k$, and define the symplectic
reflection algebra $H_{t,\mathrm{\bf{c}}} :=
H_{t,\mathrm{\bf{c}}}(V,\omega,\Gamma)$. We shall continue to
assume throughout that $$t=1.$$

\subsection{The central invariant subalgebra} \label{invar} Let $H =
H_{1,\mathrm{\bf{c}}} := H_{1,\mathrm{\bf{c}}}(\mathfrak{h}\oplus
\mathfrak{h}^*,\omega,\Gamma)$ be a rational Cherednik algebra.
Notice that the skew group algebras $S(\mathfrak{h})\ast \Gamma$
and $S(\mathfrak{h}^*)\ast \Gamma$ are contained in $H$, as is
clear from the defining relations of $H$ and the
Poincar$\acute{\mathrm{e}}$-Birkhoff-Witt theorem (\ref{PBW}) .
Hence the  centres of these algebras, $S(\mathfrak{h})^{\Gamma}$
and $S(\mathfrak{h}^*)^{\Gamma}$, are also in $H$. Less trivially,
it is proved in \cite[Proposition 4.15]{EG} using Dunkl operators
that, when $t=0$ and $k$ has characteristic 0, \begin{eqnarray}
S(\mathfrak{h})^{\Gamma} \otimes S(\mathfrak{h}^*)^{\Gamma}
\subseteq Z(H); \label{grunge}
\end{eqnarray} an alternative proof of the same result is given in
\cite[Proposition 3.6]{G}. When $t = 1$ and the characteristic is
positive, $H_{1,\mathbf{0}}$ is  $D(\mathfrak{h})\ast \Gamma$, the
skew group algebra of the Weyl algebra, so its centre is easily
calculated to be $(S(\mathfrak{h} \oplus
\mathfrak{h}^*)^{\operatorname{p}})^{\Gamma}$. It's obvious,
therefore, what the characteristic $p$ analogue of (\ref{grunge})
should be; we offer here a proof of the result which is completely
elementary and can be adapted to give also an easy proof of the
original characteristic 0 theorem.

\begin{prop} Let $H = H_{1, \mathbf{c}}(\mathfrak{h}\oplus \mathfrak{h}^*,\omega,\Gamma)$ be a rational Cherednik
algebra. Then
$$ Z_0 := (S(\mathfrak{h})^{\operatorname{p}})^{\Gamma}\otimes (S(\mathfrak{h}^*)^{\operatorname{p}})^{\Gamma}
\subseteq Z(H).$$
\end{prop}
\begin{proof} Fix a pair of dual bases $\{x_i\}$ and $\{y_i\}$ of
$\mathfrak{h}$ and $\mathfrak{h}^*.$ Thus the filtration
$\{\mathcal{F}_i\}$ of (\ref{prop}) is defined by setting
degree$(x_i)=$ degree$(y_i)=1$ and degree$(\gamma)=0,$ for $1 \leq i
\leq n$ and $\gamma \in \Gamma.$ Note that $H$ is also a
$\mathbb{Z}-$graded algebra, with degree$(x_i)=1,$ degree$(y_i)=
-1,$ and degree$(\gamma)=0,$ for $1 \leq i \leq n$ and $\gamma \in
\Gamma.$ Let's denote the $i$th graded subspace by
$\mathcal{M}_i(H),$ so $H = \oplus_{i \in \mathbb{Z}}
\mathcal{M}_i(H),$ and, putting $\mathcal{M}_i(Z) :=
\mathcal{M}_i(H) \cap H,$
\begin{eqnarray} Z := Z(H) = \oplus_{i\in \mathbb{Z}}
\mathcal{M}_i(Z).\label{zgrad}\end{eqnarray}

Let $f$ be a homogeneous element of
$(S(\mathfrak{h})^{\operatorname{p}})^{\Gamma}$ of degree $m$. We
aim first to show that
\begin{eqnarray} f \in Z. \label{frobinv}\end{eqnarray}
Let $\sigma_m^{\mathcal{F}} : \mathcal{F}_mZ \longrightarrow
\mathcal{F}_mZ/\mathcal{F}_{m-1}Z$ be the symbol map of degree
$m.$ Since $f \in S(\mathfrak{h}),$ its $\mathcal{F}-$degree and
$\mathcal{M}-$degree are the same, namely
$$ \mathrm{degree}_{\mathcal{F}}(f) =  \mathrm{degree}_{\mathcal{M}}(f) = m.$$
Note that $(S(V)^{\Gamma})^{\operatorname{p}} \cong
(S(V)^{\operatorname{p}})^{\Gamma},$ by the $\Gamma -$equivariance
of the Frobenius homomorphism. Thus, by Theorem \ref{centre}(ii)(a),
there exists $z \in \mathcal{F}_m Z$ with $\sigma_m^{\mathcal{F}}(z)
= f.$ On the other hand, by (\ref{zgrad}), we can write
$$z = z_1 + z_2, $$
where $z_1 \in \mathcal{M}_m(Z)$ and $z_2 \in \bigoplus_{j \neq
m}\mathcal{M}_j(Z).$ Now, by the
Poincar$\acute{\mathrm{e}}$-Birkhoff-Witt theorem,
(\ref{prop})(\ref{PBW}), there is \emph{no} cancellation between
$\sigma_m^{\mathcal{F}}(z_1)$ and $\sigma_m^{\mathcal{F}}(z_2).$
That is,
$$ f = \sigma_m^{\mathcal{F}}(z) = \sigma_m^{\mathcal{F}}(z_1) + \sigma_m^{\mathcal{F}}(z_2).$$
Since $\mathcal{M}-$degree$(f) = m,$
$$ z_1 = f + g,$$
where $g \in \mathcal{M}_m(H).$ We claim that $g = 0.$ Suppose $g
\neq 0.$ If there is a monomial $\mathbf{x}^I$ in $\{x_1, \ldots ,
x_n \}$, with $I = (m_i) \in \mathbb{Z}_{\geq 0}^n$, and $\gamma
\in \Gamma$, such that $u := \mathbf{x}^I\gamma$ occurs in the PBW
expression for $g$ with non-zero coefficient, then $u$ would
appear in $\sigma_m^{\mathcal{F}}(z_1)$, and could not be
cancelled by any term from $\sigma_m^{\mathcal{F}}(z_2)$ since
$\mathcal{M}-$degree$(u)= m.$ This would contradict the fact that
$\sigma_m^{\mathcal{F}}(z)=f.$ Hence every basis term $u$ in $g$
has the form $\mathbf{x}^I\mathbf{y}^J\gamma$ with $J \neq (0,
\ldots , 0).$ But $g \in \mathcal{M}_m(H),$ so that $|I| - |J| =
m,$ where $|I| = \sum_i m_i.$ Therefore $|I| > m$, forcing
$\mathcal{F}-$degree$(u)>m.$ However this contradicts $z \in
\mathcal{F}_mZ,$ and so $g=0.$ Therefore $z = z_1 + z_2 = f +
z_2.$ That is, $f = z_1 \in \mathcal{M}_m(Z),$ so (\ref{frobinv})
is proved.

Since $(S(\mathfrak{h})^{\operatorname{p}})^{\Gamma}$ is generated
by homogeneous elements, it follows that $\\
(S(\mathfrak{h})^{\operatorname{p}})^{\Gamma} \subseteq Z$; and, by
swapping the grading on $\mathfrak{h}$ and $\mathfrak{h}^*$, we can
show in the same way that
$(S(\mathfrak{h}^*)^{\operatorname{p}})^{\Gamma}\subseteq Z.$
Finally, the Poincar$\acute{\mathrm{e}}$-Birkhoff-Witt theorem
implies that the subalgebra of $Z$ generated by these two invariant
rings is $(S(\mathfrak{h})^{\operatorname{p}})^{\Gamma} \otimes
(S(\mathfrak{h}^*)^{\operatorname{p}})^{\Gamma}$, completing the
proof of the proposition.
\end{proof}

\noindent \textbf{Remark.} Keep the notation of the proposition. By
the Shepherd-Todd-Chevalley theorem, \cite[Theorem 7.2.1]{Be}, the
central subalgebra $Z_0$ is a polynomial algebra in $2n$
indeterminates. Moreover, by classical invariant theory
 $S(\mathfrak{h})$ [resp. $S(\mathfrak{h}^*)$] is a
free $(S(\mathfrak{h})^{\operatorname{p}})^{\Gamma}-$module [resp.
$(S(\mathfrak{h}^*)^{\operatorname{p}})^{\Gamma}-$module] of rank
$p^n |\Gamma|.$ Thus, by the
Poincar$\acute{\mathrm{e}}$-Birkhoff-Witt theorem,
\begin{eqnarray} H \textit{ is a free } Z_0-\textit{module of rank }
p^{2n}|\Gamma|^3. \label{free} \end{eqnarray} Hence there is a
bundle $\mathcal{B}$ of algebras of $k-$dimension $p^{2n}|\Gamma|^3$
over affine $2n-$space, and every irreducible $H-$module is a module
for precisely one of the algebras in $\mathcal{B}.$ Thus it makes
sense to study the representation theory of $H$ by studying
$\mathcal{B}.$
\medskip

\noindent \textbf{Example.} \emph{Kleinian singularities of Type
A.} Let $r \in \mathbb{Z}$, $r > 1,$ with $r$ coprime to $p,$ and
let $\eta$ be a primitive $r$th root of 1 in $k.$ Let
$\mathfrak{h} = kx,$ $\mathfrak{h}^* = ky,$ and let $\Gamma =
\langle \gamma \rangle$ be the cyclic group of order $r$ acting on
$\mathfrak{h}$ by $\gamma.x = \eta x,$ so that $\gamma.y =
\eta^{-1} y.$ Thus, for $\mathbf{c} = (c_1, \ldots , c_{r-1}) \in
k^r ,$ $H = H_{1, \mathbf{c}}(\mathfrak{h}\oplus \mathfrak{h}^*,
\omega, \Gamma)$ is the algebra
\begin{eqnarray} k \langle x,y,\gamma : & \gamma x = \eta x \gamma,
\; \gamma y = \eta^{-1} y \gamma, \\ \nonumber & [y,x] = 1 -
\sum_{j=1}^{r-1} c_j \gamma^j \rangle. \label{Klein}
\end{eqnarray}
One checks easily that $$ x^{pr} \in Z(H), \quad y^{pr} \in Z(H).
$$ It is convenient to have available also the basis of $k\Gamma$
afforded by the primitive idempotents $e_j :=
\frac{1}{r}\sum_{i=0}^{r-1}\eta^{ij}\gamma^i ,$ for $j= 0, \ldots
, r-1,$ with respect to which we write the commutator relation as
$$ [y,x] = \sum_{j=0}^{r-1}f_je_j, $$ for $\mathbf{f} = (f_j) \in
k^n$ with $\sum_j f_j = r.$ The interested reader may write down
the linear relations between the $c_i$ and the $f_j.$ Define
$$ \tau = xy + \sum_{i=1}^{r-1}(i - \sum_{j=0}^{i-1}f_j)e_i \in H.
$$
Then $[\tau, x] = x$ and $[\tau, y] = -y,$ so that $$ h := \tau^p
- \tau \in Z(H). $$ Define elements $\{\delta_m : 0 \leq m \leq
r-1 \}$ of $k$ by
$$ \delta_m := \sum_{j=1}^{r-1}c_j(1-\eta^{-j})^{-1}\eta^{mj} =
-\frac{1}{r}\sum_{l=0}^{r-1}(\rho_{m,l+1})f_l,$$ where
$\rho_{m,l+1} := \sum_{j=1}^{r-1}[\eta^{(m+l)j}/(\eta^j -  1)].$
The following result is proved by direct calculation in
\cite[Chapter 3]{Noi}.
\begin{prop} Let $H = H_{1,\mathbf{c}}$ be a symplectic reflection
algebra for the Kleinian singularity of type A, as defined above.
Keep the notation as above.
\begin{enumerate}\item $Z(H) = k \langle x^{pr}, y^{pr}, \tau
\rangle.$ \item $ Z(H) \cong k[X,Y,Z : XY = \Pi_{m=0}^{r-1}(Z +
\delta_m^p - \delta_m)].$ \item $Z(H)$ is smooth if and only if
$\delta_i - \delta_j \in \mathbb{Z}$ only when $i = j.$
\end{enumerate}

\end{prop}

\subsection{The Dunkl embedding} \label{Dunkl}  One reason why rational Cherednik algebras over $\mathbb{C}$ at $t=1$ are easier to
handle than arbitrary symplectic reflection algebras is that the
$\mathbb{C}-$algebra $H_{\mathbb{C}} :=
H_{1,\mathrm{\bf{c}}}(\mathfrak{h}\oplus
\mathfrak{h}^*,\omega,\Gamma)$ embeds in the skew group algebra
$\mathcal{D}(\mathfrak{h}^{\mathrm{reg}})\ast \Gamma$; indeed
$H_{\mathbb{C}}$ is birationally equivalent to the skew group
algebra  $\mathcal{D}(\mathfrak{h})\ast \Gamma$ over the Weyl
algebra $\mathcal{D}(\mathfrak{h}),$ \cite[Proposition 4.5]{EG}.

In fact, the same is true in positive characteristic, with
essentially the same proof. We keep the notation as in
(\ref{Cherdef}). In addition, for each pseudo-reflection $s \in S
\subseteq \Gamma$ choose $\alpha_s \in \mathfrak{h}^*$ whose kernel
is the hyperplane stabilized by $s$, and taking $\Gamma-$conjugates
as appropriate so that $\delta := \Pi_{s \in S} \alpha_s \in
S(\mathfrak{h}^*)^{\Gamma}.$ Write $\mathfrak{h}^{\mathrm{reg}}$ for
the regular points of $\mathfrak{h}$, so that
$\mathcal{D}(\mathfrak{h}^{\mathrm{reg}}),$ the algebra of
differential operators on $\mathfrak{h}^{\mathrm{reg}},$ is just
$\mathcal{D}(\mathfrak{h})[\delta]^{-1}$, the localisation of the
$n$th Weyl algebra $\mathcal{D}(\mathfrak{h}) = k[y_1, \ldots ,
y_n,\partial/\partial y_1, \ldots ,\partial/\partial y_n] $ at the
powers of $\delta .$

\begin{thrm} Let $H$ be a rational Cherednik algebra, with notation
as above. There are elements $\tau_1, \ldots , \tau_n$ of $k\Gamma$
such that the assignment $\gamma \mapsto \gamma,$ $y_i \mapsto y_i$
and $x_i \mapsto \partial/\partial y_i + \tau_i$ for $\gamma \in
\Gamma$ and $i = 1, \ldots , n$ extends to an injective algebra
homomorphism $\Theta_{\textbf{c}} : H \longrightarrow
\mathcal{D}(\mathfrak{h}^{\mathrm{reg}})\ast \Gamma.$
\end{thrm}
\begin{proof} The proof of \cite[Proposition 4.5]{EG} works equally
well in positive characteristic.
\end{proof}

\noindent {\bf Remark:} In fact, the proof yields a stronger
statement: $\Theta_{\textbf{c}}$ becomes an isomorphism after
inverting $\delta$; that is, $H[\delta^{-1}] \cong
\mathcal{D}(\mathfrak{h}^{\mathrm{reg}})\ast \Gamma.$

\subsection{Goldie rank} \label{goldie} Recall \cite[Theorem 10.4.4]{MR} that the
\emph{Goldie} or \emph{uniform rank} $\mathrm{udim}_R(M)$ of a
module $M$ over the noetherian ring $R$ is the biggest number of
non-zero modules whose direct sum embeds into $M$, or infinity if no
such supremum exists. If $\mathcal{C}$ is an Ore set of
non-zero-divisors in $R$ and $M$ is $\mathcal{C}-$torsion free then
it is easy to check that \begin{eqnarray} \mathrm{udim}_R(M) =
\mathrm{udim}_{R\mathcal{C}^{-1}}(M\otimes_R R\mathcal{C}^{-1}),
\label{udim} \end{eqnarray} \cite[proof of Lemma 10.2.13]{MR}. Apply
this in particular when $R$ is a prime noetherian ring,
$\mathcal{C}$ is the set of \emph{all} non-zero-divisors in $R$, and
$M=R$. In this case $\mathcal{C}$ is an Ore set and
$R\mathcal{C}^{-1}$ is the simple artinian quotient ring $Q(R)$ of
$R$, by Goldie's theorem, \cite[Theorem 10.4.10]{MR}. Thus $Q(R)
\cong M_s(D)$ for a division ring $D$ by the Artin-Wedderburn
theorem, and (\ref{udim}) shows that the integer $s$ is the Goldie
rank of $R$.

\begin{thrm} Let $H = H_{1,\mathbf{c}}(\mathfrak{h},\omega,\Gamma)$ be a
rational Cherednik algebra.
\begin{enumerate} \item $\mathrm{udim}(H) = |\Gamma|.$
\item The integer $s$ appearing in Theorem \ref{return} is the
Goldie rank of $H$, and so equals $|\Gamma|$.\end{enumerate}
\end{thrm}

\begin{proof} Fix a prime $\mathfrak{p}$ of $Z(H)$ which is in the Azumaya locus. Then
the isomorphism of Theorem \ref{return}(ii) holds, and since
$eH_{\mathfrak{p}}e$ is a domain by Theorem (\ref{base})(ii), the
first claim in (ii) follows. Thus it remains to prove that
$\mathrm{udim}(H) = |\Gamma|.$ Since the Goldie rank of an algebra
is unaltered by inverting an Ore set of non-zero-divisors, by
(\ref{udim}), in view of Theorem (\ref{Dunkl}) we only need to show
that
$$\mathrm{udim}(\mathcal{D}(\mathfrak{h})\ast \Gamma) = |\Gamma |.
$$
This follows from a special case of Moody's theorem, \cite[Theorem
37.14]{Pass}.
\end{proof}

\noindent {\bf Remark:} Presumably Theorem \ref{goldie}(i) is true
for all symplectic reflection algebras over all fields. Suppose
first that $K$ is algebraically closed of characteristic 0. Then the
Dunkl embedding can be used as above to deal with rational Cherednik
$K-$algebras when $t=1.$ Secondly, if $K$ has \emph{any}
characteristic and $H$ is \emph{any} symplectic reflection algebra
with $t=0,$ then $eHe$ is commutative by \cite[Theorem 3.1]{EG}, so
the integer $s$ of Theorem \ref{return}(ii) is equal to the
PI-degree of $H$. But the latter is equal to the dimension over the
algebraically closed field $k$ of a generic irreducible
representation of $H$, by \cite[Theorem III.1.6 and Lemma
III.1.2]{BG}, and this is known to be $|\Gamma|$ by \cite[Theorem
1.7(iv)]{EG}.

Thus the only case remaining open is that of a symplectic reflection
algebra which is not Cherednik, with $t=1$.

\section{Cherednik algebras - representation theory} \label{Cherrep}

\subsection{PI-degree and centre} \label{PIdeg} When $k$ has characteristic 0
and $t=0,$ the PI-degree of a symplectic reflection algebra
$H_{0,\mathbf{c}}$ is $|\Gamma|$, \cite[Theorem 1.7(iv)]{EG};
indeed the irreducible $H_{0,\mathbf{c}}-$modules in the Azumaya
locus are isomorphic as $k\Gamma-$modules to $k\Gamma.$ The same
conclusions remain valid when $t=0$ and $k$ has positive
characteristic, with essentially the same proofs. It remains to
consider the case $t=1$ when $k$ has positive characteristic; here
we deal with the Cherednik algebras. For this we need the
following two well-known facts:

\begin{lem} \label{weyl}\begin{enumerate}
\item The centre of the Weyl algebra $\mathcal{D}(\mathfrak{h}) =
k\langle x_1, \ldots , x_n, y_1, \ldots , y_n \rangle$ is $k\langle
x_1^p, \ldots , x_n^p, y_1^p, \ldots , y_n^p \rangle$; indeed
$\mathcal{D}(\mathfrak{h}) $ is Azumaya over\\ $k\langle x_1^p,
\ldots , x_n^p, y_1^p, \ldots , y_n^p \rangle$.
\item Let $R$ be a domain and let $G$ be a finite group of
automorphisms of $R$ which acts faithfully on $Z(R).$  Then the
centre of the skew group ring $R \ast G$ is  $Z(R)^{G}.$
\end{enumerate}
\end{lem}
\begin{proof} (i) is proved in \cite{Rev} and (ii) is a straightforward exercise.
\end{proof}

\begin{thrm} Let
$H=H_{1,\mathbf{c}}(\mathfrak{h}\oplus
\mathfrak{h}^*,\omega,\Gamma)$ be a rational Cherednik algebra
over an algebraically closed field $k$ of positive characteristic
$p.$ Let $\mathfrak{h}$ have dimension $n.$ Then
$$ \mathrm{PI-degree}(H) = p^{n}|\Gamma|. $$
\end{thrm}
\begin{proof} By Theorem and Remark \ref{Dunkl}, $H$ and $\mathcal{D}(\mathfrak{h})\ast
\Gamma$ become isomorphic after inverting certain central
elements, so it is enough, by \cite[Lemma 10.2.13]{MR}, to prove
that $\mathcal{D}(\mathfrak{h})\ast \Gamma$ has PI-degree
$p^{n}|\Gamma|.$ From the lemma we deduce that
$$ Z:= Z(\mathcal{D}(\mathfrak{h})\ast \Gamma)\quad = \quad k\langle x_1^p, \ldots , x_n^p, y_1^p, \ldots , y_n^p
\rangle^{\Gamma}.$$ Now, inverting the non-zero elements of the
centre in $\mathcal{D}(\mathfrak{h})\ast \Gamma$ and calculating
that
$$\mathrm{dim}_{Q(Z)}(Q(Z) \otimes_Z \mathcal{D}(\mathfrak{h})\ast
\Gamma)\quad = \quad p^{2n}|\Gamma|^2,$$ we obtain the desired
conclusion.
\end{proof}

Presumably Theorem \ref{PIdeg} is true for all symplectic reflection
algebras with $t=1$ over a characteristic $p$ field; we leave this
as an open question.

\begin{cor} Let
$H=H_{1,\mathbf{c}}(\mathfrak{h}\oplus
\mathfrak{h}^*,\omega,\Gamma)$ be a rational Cherednik algebra
over an algebraically closed field $k$ of positive characteristic
$p.$  Then $Z(H)$ is a free module of rank $|\Gamma|$ over its
polynomial subalgebra $\\ Z_0 = (S(\mathfrak{h}^p)^{\Gamma}
\otimes (S(\mathfrak{h^*}^p)^{\Gamma}.$
\end{cor}

\begin{proof} Let $\mathfrak{h}$ have dimension
$n.$ By Theorem \ref{centre}(iv), $Z(H)$ is Gorenstein, and
therefore Cohen-Macaulay. Thus, by \cite[Corollary 18.17]{Eis}, it
is free over its polynomial subalgebra $Z_0$. The rank is determined
by comparing $\mathrm{dim}_{Q(Z_0)}(Q(Z_0)\otimes H)$ with
$\mathrm{dim}_{Q(Z)}(Q(Z_0)\otimes H);$ the first of these is
$p^{2n}|\Gamma|^3$ by (\ref{free}), while the second is
$p^{2n}|\Gamma|^2$ by Theorem \ref{PIdeg}.
\end{proof}

\subsection{$\Gamma-$regularity of the generic irreducible
modules} \label{regular} It follows from Theorem \ref{PIdeg} that
the maximal $k-$dimension of the irreducible $H-$modules is
$p^n|\Gamma|$. By the structure theory of noetherian PI-rings, this
dimension is achieved precisely by those irreducible $H-$modules $V$
for which $\mathfrak{m} := \mathrm{Ann}_{Z(H)}(V)$ has the property
that $H/\mathfrak{m}H$ is simple artinian, and in this case
$H/\mathfrak{m}H \cong M_{p^n|\Gamma|}(k).$ The open set of such
$\mathfrak{m}$ is precisely the Azumaya locus of (\ref{return}).
Recall that \cite[Theorem 1.7(vi)]{EG} shows that, for \emph{any}
symplectic reflection algebra over any field, at $t=0,$ the
irreducible modules of maximal dimension are $k\Gamma$-regular of
rank one. Analogously, we can describe the $k\Gamma-$structure of
the Azumaya irreducibles over Cherednik algebras for $t \neq 0$ in
positive characteristic. We begin with an easy lemma which
essentially ensures that the desired result is true for
$H_{1,\bf{0}}.$

\begin{lem} Let $k$ have characteristic $p>0$ as usual, and let $(V,\omega)$
be a symplectic $k$-vector space with basis $\{x_1, \ldots ,
x_n,y_1, \ldots ,y_n\}$, with $\omega(x_i,y_j) = \delta_{ij},\,
\omega(x_i,x_j) =\omega(y_i,y_j)=0$. Let $\Gamma$ be a finite
subgroup of $\mathrm{Sp}(V),$ of order prime to $p$. Write
$\mathfrak{h}$ for the subspace $\sum_i kx_i$ of $V$, so that
$\Gamma$ acts by automorphisms on both $S(V)$ and
$\mathcal{D}(\mathfrak{h}).$ Let $F$ be the quotient field of
$(S(V)^{\Gamma})^{\operatorname{p}}.$ Then $F
\otimes_{(S(V)^{\Gamma})^{\operatorname{p}}} \mathcal{D}(V)$ is a
free $F\Gamma$-module of rank $p^{m}$.
\end{lem}
\begin{proof} Note first that $(S(V)^{\Gamma})^{\operatorname{p}}$
is a central $\Gamma$-invariant subalgebra of
$\mathcal{D}(\mathfrak{h}),$ so the statement of the lemma makes
sense. Moreover, as $F\Gamma$-modules, there is no difference
between $F \otimes_{(S(V)^{\Gamma})^{\operatorname{p}}}
\mathcal{D}(\mathfrak{h})$ and $F
\otimes_{(S(V)^{\Gamma})^{\operatorname{p}}} S(V),$ so we can work
with the latter. Then $F \otimes S(V)$ is simply the quotient field
of $S(V)$, which is a free $Q(S(V)^{\Gamma})\Gamma$-module of rank
one by the Primitive Element Theorem of Galois theory. Since
$\mathrm{dim}_F(Q(S(V)^{\Gamma})) = p^m,$ the result follows.
\end{proof}

\begin{thrm} Let $H$ be as in Theorem \ref{PIdeg}, and let the maximal
 ideal $\mathfrak{m}$ of $Z = Z(H)$ be in the Azumaya locus of $H$.
 \begin{enumerate} \item $He/\mathfrak{m}He$ is a regular $k\Gamma -$module of rank $p^{2n}.$
 \item The irreducible $H-$ module $V$ with $\mathrm{Ann}_Z(V) = \mathfrak{m}$ is a
regular $k\Gamma-$module of rank $p^n$.
\end{enumerate}
\end{thrm}

\begin{proof} (i) By Theorems \ref{return}(ii) and \ref{goldie}(ii), for every
prime $\mathfrak{p}$ of $Z$ in the Azumaya locus of $H$,
\begin{eqnarray} H_{\mathfrak{p}} \cong
M_{|\Gamma|}(eH_{\mathfrak{p}}e).\label{matty} \end{eqnarray} Set
$Q:= Q(Z),$ the quotient field of $Z$, so by Theorem \ref{PIdeg} we
have
$$ \mathrm{dim}_Q(Q \otimes_Z H) = p^{2n}|\Gamma|^2. $$
Thus (\ref{matty}) implies that
\begin{eqnarray} \mathrm{dim}_Q(Q \otimes_Z He) = p^{2n}|\Gamma|.
\label{Qdim} \end{eqnarray} Now let $\mathfrak{m}$ be a maximal
ideal of $Z$ in the Azumaya locus of $H$. The Azumaya property
ensures that $H_{\mathfrak{m}}$, and hence also $H_{\mathfrak{m}}e,$
are projective and thus free $Z_{\mathfrak{m}}-$modules. In
particular, from (\ref{Qdim}),
\begin{eqnarray} H_{\mathfrak{m}}e \textit{ is }
Z_{\mathfrak{m}}\textit{-free of rank } p^{2n}|\Gamma|. \label{free}
\end{eqnarray} Now let $\mathrm{Irr}(k\Gamma)$ be the set of
isomorphism classes of irreducible $k\Gamma -$modules. We
decompose $He$ as the direct sum of its isotypic components as
left $k\Gamma-$module:
\begin{eqnarray} He = \bigoplus_{E \in
\mathrm{Irr}(k\Gamma)}\mathrm{Isot}_E(He). \label{decompo}
\end{eqnarray}
Of course this is a sum of $Z-$modules as well as
$k\Gamma-$modules, so applying $Q \otimes_Z -$ to (\ref{decompo})
yields $$ Q \otimes_Z He = \bigoplus_{E \in
\mathrm{Irr}(k\Gamma)}(Q \otimes_Z \mathrm{Isot}_E(He)).
$$
Thanks to the Dunkl embedding, Theorem and Remark \ref{Dunkl}, $H$
is birationally equivalent to the skew group algebra
$\mathcal{D}(\mathfrak{h})\ast\Gamma,$ via a map which is the
identity when restricted to $k\Gamma$. By this and the above lemma,
\begin{eqnarray} Q \otimes_Z He \textit{ is }
Q\Gamma\textit{-regular of rank } p^{2n}. \label{birat}
\end{eqnarray}
By (\ref{free}), the localised isotypic components
$$ Z_{\mathfrak{m}} \otimes_Z \mathrm{Isot}_E(He) \cong
\mathrm{Isot}_E(H_{\mathfrak{m}}e) $$ are $Z_{\mathfrak{m}}-$free
for each $k\Gamma-$irreducible $E$; and so, in view of
(\ref{birat}), $$ Z_{\mathfrak{m}} \otimes_Z \mathrm{Isot}_E(He)
\textit{ has } Z_{\mathfrak{m}}\textit{-rank }
p^{2n}(\mathrm{dim}_k(E))^2. $$ We deduce from this that,
factoring $Z_{\mathfrak{m}}$ and the isotypic component by
$\mathfrak{m}Z_{\mathfrak{m}},$ $$
\mathrm{dim}_k(\mathrm{Isot}_E(He/\mathfrak{m}He) =
p^{2n}(\mathrm{dim}_k(E))^2. $$ That is, the multiplicity of $E$
in $He/\mathfrak{m}He$ is $p^{2n}\mathrm{dim}_k(E),$ proving (i).

(ii) By Theorem \ref{PIdeg}, $\mathrm{dim}_k(V) = p^n|\Gamma|$.
Since $V$ is the unique irreducible module for the simple artinian
ring $H/\mathfrak{m}H$, $He/\mathfrak{m}He$ is the sum of $p^n$
copies of $V$. Therefore it follows from (i) that $V$ is
$k\Gamma-$regular of rank $p^n.$
\end{proof}

\subsection{Azumaya versus smooth locus} \label{Azumaya}

\begin{thrm} Let $H=H_{1,\mathbf{c}}(\mathfrak{h},\omega,\Gamma)$ be a
rational Cherednik algebra over an algebraically closed field $k$
of positive characteristic $p.$ Let $Z$ be the centre of $H$, and
write $\mathcal{A}_H$ for the Azumaya locus of $Z$ in $H$, and
$\mathcal{S}_Z$ for the singular locus of $\mathrm{Maxspec}(Z).$
Then
$$ \mathcal{A}_H = \mathrm{Maxspec}(Z) \setminus \mathcal{S}_Z. $$
\end{thrm}
\begin{proof} Since $H$ is a finite $Z$-module and is
Auslander-regular and Cohen-Macaulay, by Theorems \ref{centre}(ii)
and \ref{prop}(iii), it follows from \cite[Theorem 3.8]{BGood} that
it's enough to prove that $H$ is Azumaya in codimension one. That
is, let $\mathfrak{p}$ be a prime ideal of $Z$ of height one. We
must show that $H_{\mathfrak{p}}$ is Azumaya; equivalently we must
exhibit a maximal ideal $\mathfrak{m}$ of $Z$ with $\mathfrak{m}$
Azumaya and $\mathfrak{p} \subseteq \mathfrak{m}.$ Let $Z_0$ be the
polynomial subalgebra $(S(\mathfrak{h})^{\operatorname{p}})^{\Gamma}
\otimes (S(\mathfrak{h}^*)^{\operatorname{p}})^{\Gamma}$ of $Z$
provided by Proposition \ref{invar}. Thus $\mathfrak{p}_0 :=
\mathfrak{p} \cap Z_0$ is a prime ideal of $Z_0$, and
\begin{eqnarray} \mathrm{height}(\mathfrak{p}_0) = 1, \label{height}
\end{eqnarray} by Lying Over, \cite[Proposition 4.15]{Eis}. We claim
that
\begin{eqnarray} \textit{either } \mathfrak{p}_0 \cap
(S(\mathfrak{h})^{\operatorname{p}})^{\Gamma} = 0 \textit{ or }
\mathfrak{p}_0 \cap (S(\mathfrak{h}^*)^{\operatorname{p}})^{\Gamma}
= 0. \label{either}
\end{eqnarray}
For, suppose for a contradiction that both intersections are
non-zero. Then $$\mathfrak{q} := (\mathfrak{p}_0 \cap
(S(\mathfrak{h})^{\operatorname{p}})^{\Gamma})Z_0 =
(\mathfrak{p}_0 \cap
(S(\mathfrak{h})^{\operatorname{p}})^{\Gamma})\otimes
(S(\mathfrak{h}^*)^{\operatorname{p}})^{\Gamma}$$ is a non-zero
prime of $Z_0$ contained in $\mathfrak{p}_0,$ and clearly
$\mathfrak{q} \cap (S(\mathfrak{h}^*)^{\operatorname{p}})^{\Gamma}
= 0,$ so that $\mathfrak{q} \subsetneqq \mathfrak{p}_0.$ But this
contradicts (\ref{height}), and hence (\ref{either}) is true.

Let's suppose first that $\mathfrak{p}_0 \cap
(S(\mathfrak{h}^*)^{\operatorname{p}})^{\Gamma} = 0 .$ Then, in
particular, $\mathfrak{p}$ does \emph{not} contain the element
$\delta$ defined in (\ref{Dunkl}), and hence there is a maximal
ideal $\mathfrak{m}$ of $Z$ with $\mathfrak{p} \subseteq
\mathfrak{m}$, such that $\mathfrak{m}$ does not contain $\delta.$
We claim that $\mathfrak{m}$ is Azumaya; in view of \cite[Theorem
III.1.6]{BGb} and Theorem \ref{PIdeg}, this amounts to showing that,
if $W$ denotes an irreducible $H-$module killed by $\mathfrak{m},$
then
\begin{eqnarray} \textrm{dim}_k(W) = p^n|\Gamma|. \label{big}
\end{eqnarray}
Now $\delta$ acts as multiplication by a non-zero scalar on $W$;
so, since $H[\delta]^{-1} \cong
\mathcal{D}(\mathfrak{h}^{\mathrm{reg}})\ast \Gamma$ by Theorem
and Remark \ref{Dunkl}, $W$ admits actions of (i)
$\mathcal{D}(\mathfrak{h}^\textrm{reg}),$ and of (ii)
$S(\mathfrak{h}^*)[\delta]^{-1}\ast \Gamma$.

From (i) and Lemma \ref{PIdeg}(i) we deduce that \begin{eqnarray}
p^n\mid \textrm{dim}_k(W). \label{pdiv}\end{eqnarray} Let $U$ be
any irreducible $S(\mathfrak{h}^*)[\delta]^{-1}\ast
\Gamma-$module, so $\mathrm{dim}_k(U) < \infty$ and so there is a
maximal ideal $\mathfrak{t}$ of $S(\mathfrak{h}^*)[\delta]^{-1}$
and $0 \neq u \in U$ with $\mathfrak{t}u = 0.$ Set $U_1 :=
\mathrm{Ann}_U(\mathfrak{t}),$ a non-zero
$S(\mathfrak{h}^*)[\delta]^{-1}-$submodule of $U$. For each
$\gamma \in \Gamma$, $\gamma U_1 =
\mathrm{Ann}_U(\mathfrak{t}^{\gamma})$ is isomorphic as a vector
space to $U_1$. Now $\mathfrak{t}$ has $|\Gamma|$ distinct
$\Gamma-$conjugates, by definition of $\delta$. Consider $U' :=
\sum_{\gamma \in \Gamma} \gamma U_1 \subseteq U.$ Clearly $U'$ is
a non-zero $S(\mathfrak{h}^*)[\delta]^{-1}\ast \Gamma-$submodule
of $U$, and therefore $U' = U$. Moreover the sum in the definition
of $U'$ has $|\Gamma|$ distinct terms, each term killed by a
distinct maximal ideal. So the sum is direct, and hence
$\mathrm{dim}_k(U) = |\Gamma|\mathrm{dim}_k(U_1).$ In particular,
$|\Gamma|\mid \mathrm{dim}_k(U)$; since $W$ has a finite
composition series as $S(\mathfrak{h}^*)[\delta]^{-1}\ast
\Gamma-$module,
\begin{eqnarray} |\Gamma| \mid \mathrm{dim}_k(W). \label{Gdiv}
\end{eqnarray}
Combining (\ref{pdiv}) and (\ref{Gdiv}), recalling that $p \nmid
|\Gamma|$ by hypothesis, proves (\ref{big}), and so the theorem
follows.
\end{proof}

\section{Questions and conjectures} \label{qn} Throughout, $H =
H_{1,\mathbf{c}}(V,\omega,\Gamma)$ is a symplectic reflection
algebra over $k$, which is algebraically closed of characteristic
$p > 0.$ Let $V$ have dimension $2n.$

We repeat the question about Goldie ranks which was stated, with
background discussion, in Remark \ref{goldie}:

\noindent {\bf Question A:} Does $H$ have Goldie dimension
$|\Gamma|$?\footnote{We understand that Iain Gordon has been able to
confirm this.}
\medskip

Similarly, it seems reasonable to expect that the value for the
PI-degree of Cherednik algebras obtained in (\ref{PIdeg}) applies
in general:

\noindent {\bf Question B:} Does $H$ have PI-degree $p^n|\Gamma|$?
\medskip

A more precise version of the above question is:

\noindent {\bf Question C:} Is every simple $H$-module of maximal
dimension a regular $k\Gamma -$module of rank $p^n$?
\medskip

It is of interest from the perspective of noncommutative resolutions
of singularities to ask:

\noindent {\bf Question D:} For which $H$ do there exist values of
the parameter $\mathbf{c}$ for which $\mathrm{Maxspec}(Z(H))$ is
smooth? When such values exist determine them all.
\medskip

The analogue of the first part of Question D in characteristic 0 at
$t=0$ has been answered completely, as a result of a considerable
body of work - see \cite{EG,G,Mar,Bel}. A natural strategy to attack
this problem in the Cherednik case is afforded by Theorem
\ref{Azumaya}. To have this route available in the setting of an
arbitrary symplectic reflection algebra, one needs therefore to
answer:

\noindent {\bf Question E:} Does the Azumaya locus coincide with
the smooth locus for an arbitrary symplectic reflection algebra?
\medskip

Work on the finite dimensional representation theory in
characteristic 0 is considerably helped by the underlying Poisson
structure - in view of \cite[Theorems 4.2 and 7.8]{BG}, there are
only finitely many symplectic leaves in $\mathrm{Maxspec}(Z(H))$,
and the representation theory is constant across leaves, in the
sense that, if $\mathfrak{m}$ and $\mathfrak{n}$ belong to the
same leaf, then $H/\mathfrak{m}H \cong H/\mathfrak{n}H.$ This
motivates:

\noindent {\bf Question F:} Are there only finitely many
isomorphism classes of factors $H/\mathfrak{m}H$ as $\mathfrak{m}$
ranges through $\mathrm{Maxspec}(Z(H))$?
\medskip

The annoying gap in the equivalences of Theorem \ref{return} is
one indication that the symmetrising subalgebra is not very well
understood. We therefore ask:

\noindent {\bf Question G:} Is every localizable prime ideal $P$
of $H$ generated by its intersection with $Z(H)$?

\end{document}